\documentclass[11pt]{article}

\usepackage{amsmath,amsthm,amssymb,enumitem}
\usepackage{tikz}
\usetikzlibrary{arrows.meta,positioning}
\usepackage[a4paper,margin=1in]{geometry}
\usepackage[colorlinks=true,linkcolor=blue,citecolor=blue,urlcolor=blue]{hyperref}

\newtheorem{theorem}{Theorem}[section]
\newtheorem{proposition}[theorem]{Proposition}
\newtheorem{corollary}[theorem]{Corollary}
\newtheorem{problem}[theorem]{Problem}
\newtheorem{conjecture}[theorem]{Conjecture}
\newtheorem{definition}[theorem]{Definition}
\newtheorem{remark}[theorem]{Remark}
\newtheorem{example}[theorem]{Example}
\newtheorem{lemma}[theorem]{Lemma}

\newcommand{\R}{\mathbb R}
\newcommand{\A}{\mathcal A}
\newcommand{\conf}{\operatorname{conf}}
\newcommand{\Hol}{\operatorname{Hol}}
\newcommand{\loc}{\mathrm{loc}}
\newcommand{\cyc}{\mathrm{cyc}}
\newcommand{\Ftwo}{\mathbb F_2}
\newcommand{\rot}{\operatorname{rot}}
\newcommand{\parnum}{\operatorname{pt}}

\title{Combinatorics of Inflection Points of Plane Curve Shadows}
\author{Boris Shapiro\thanks{Department of Mathematics, Stockholm University,
SE--106 91 Stockholm, Sweden. Email: \texttt{shapiro@math.su.se}.}}

\date{}

\begin{document}

\maketitle

\begin{abstract}
We study the minimum number of inflection points among generic immersed closed plane curves with a fixed embedded shadow.  The word immersed is essential: a genuinely embedded Jordan curve has inflection minimum zero.  For tree-like shadows, inflection criterion converts inflection-free realizability into a finite coorientation problem on the building polygons of the shadow, see \cite{ShapiroTree}.  We sharpen this viewpoint into an exact finite formula for the minimum number of normalized inflections and record a dynamic-programming computation on the block tree.  We then push the method beyond the tree-like case.  For every embedded shadow the same coorientation model gives a universal lower bound.  For a natural larger class, called tree--necklace shadows, in which the non-tree-like blocks are separated annular cycles, the lower bound is exact after imposing an explicit $\mathbb Z_2$ holonomy condition around each necklace.  We also record the algorithmic status of the exact minimization problem and formulate a likely NP-hardness problem for unrestricted shadows.  Finally, we introduce a related invariant: the minimum possible least multiplicity of the Gauss map, equivalently the smallest guaranteed number of oriented parallel tangencies.  This ``parallel-tangent load'' is controlled by the same inflection folds but is not determined by their number alone.
\end{abstract}

\makeatletter
\insert\footins{%
  \reset@font\footnotesize
  \interlinepenalty\interfootnotelinepenalty
  \splittopskip\footnotesep
  \splitmaxdepth\dp\strutbox
  \floatingpenalty\@MM
  \hsize\columnwidth
  \@parboxrestore
}
\makeatother

\section{Introduction}

Let
\[
 c:S^1\to \R^2
\]
be a smooth immersed closed plane curve.  We call $c$ \emph{generic} if its only singularities are transverse double points.  An \emph{inflection point} of $c$ is a point at which the signed curvature vanishes, or equivalently, where the convex side of the curve changes.

For a fixed isotopy class $[c]$ of generic immersed curves, set
\[
 \mu([c])=\min_{c'\in[c]}\#\{\text{inflection points of }c'\}.
\]
The general problem of estimating or computing $\mu([c])$ is classical and is naturally related to Whitney's regular homotopy theory and Arnold's invariants of plane curves and caustics \cite{Whitney,ArnoldCaustics,ArnoldPlaneCurves}.  Tree-like curves were introduced and studied from a combinatorial viewpoint by Aicardi \cite{Aicardi}; the corresponding inflection problem was treated by the author in \cite{ShapiroTree}.  A different extension, called almost tree-like curves, was later considered in connection with Arnold invariants and alternating knots \cite{Dagit}.  The present note is narrower but more quantitative: it asks for the exact minimum number of inflections in classes slightly larger than the tree-like class.  We also keep in mind the surrounding literature on nondegenerate and locally convex curves, for example \cite{Little,Saldanha}. 
For embedded plane curves  the problem is trivial, see Figure~\ref{fig:oval-shadow}.

\begin{figure}[t]
\centering
\begin{tikzpicture}[scale=.82,>=Stealth]
\begin{scope}
\node at (0,1.65) {embedded oval};
\draw[very thick] (0,0) ellipse (1.55 and .9);
\draw[-{Stealth[length=2mm]},gray] (1.55,0) arc (0:38:1.55 and .9);
\node at (0,-1.35) {$\mu=0$};
\end{scope}
\begin{scope}[xshift=5.2cm]
\node at (0,1.65) {immersed shadow};
\draw[very thick] (-1.25,0) .. controls (-1.25,1.05) and (-.15,1.05) .. (0,0);
\draw[very thick] (0,0) .. controls (.15,-1.05) and (1.25,-1.05) .. (1.25,0);
\draw[very thick] (1.25,0) .. controls (1.25,1.05) and (.15,1.05) .. (0,0);
\draw[very thick] (0,0) .. controls (-.15,-1.05) and (-1.25,-1.05) .. (-1.25,0);
\fill (0,0) circle (1.7pt);
\node at (0,-1.35) {one transverse double point};
\end{scope}
\end{tikzpicture}
\caption{The problem is trivial for embedded Jordan curves, but nontrivial for embedded shadows of immersed curves.}
\label{fig:oval-shadow}
\end{figure}
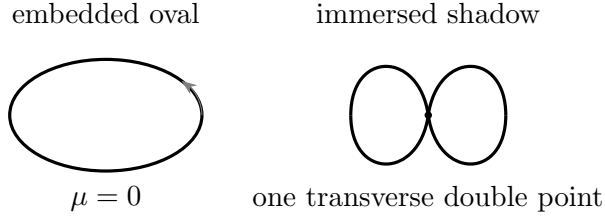

Thus the nontrivial problem concerns \emph{embedded shadows}: connected embedded planar $4$-regular graphs, together with the cyclic order of traversal induced by a generic immersion of $S^1$.  We shall use the language of curves and shadows interchangeably, but the reader should keep in mind that the curve is immersed while its shadow is embedded as a graph.

The main point of this note is the following.  For a tree-like shadow, the obstruction to removing all inflections is completely encoded by a finite coorientation problem on the pieces of the associated planar $4$-graph.  Moreover, after allowing normalized inflections at double points, the minimal number of inflections becomes an explicit finite minimization problem.  Beyond the tree-like case, cycles in the block-adjacency graph create global compatibility conditions.  We show that these conditions can be separated from the local inflection conditions and, for an annular extension of the tree-like class, encoded by a simple holonomy parity.

The paper is organized as follows.  Sections~2--6 repackage inflection tree-like theorem of \cite{ShapiroTree} as a minimization problem.  Sections~7--10 introduce the local lower bound, the annular holonomy obstruction and the exact tree--necklace formula.  Section~11 discusses the complexity of exact minimization.  Section~12 introduces the Gauss-map multiplicity, or parallel-tangent, invariant.  The final section records the remaining obstruction for general shadows.

\section{Generic curves and tree-like shadows}

Let $c$ be a connected generic immersed plane curve.  Its double points are vertices and the components of the complement of the vertices in $c$ are edges.  Thus $c$ determines an embedded planar $4$-regular graph, also called the \emph{shadow} of $c$.

\begin{definition}
A generic curve $c$ is called \emph{tree-like} if removing any double point disconnects the curve.  Equivalently, every vertex of the associated $4$-graph is a cut vertex.
\end{definition}

The terminology is justified by the following elementary observation.

\begin{lemma}
For a tree-like curve, the adjacency graph of the elementary pieces obtained by cutting the curve at all double points is a tree.
\end{lemma}

\begin{proof}
If the adjacency graph contained a cycle, then every double point corresponding to an edge of that cycle would lie on two disjoint routes between the same pieces.  Removing that double point would therefore not disconnect the curve, contradicting the definition of tree-likeness.  Conversely, the absence of such cycles is exactly the assertion that each double point separates the curve into two parts.
\end{proof}

We shall call the elementary pieces of this decomposition \emph{building polygons}.  A building polygon with $k$ sides will be called a building $k$-gon.  The sides of these polygons are not required to be straight; the terminology only records their incidence combinatorics.

\begin{figure}[t]
\centering
\begin{tikzpicture}[scale=.9,>=Stealth]
\begin{scope}
\node at (0,1.75) {a tree of blocks};
\node[draw,rounded corners,minimum width=.9cm,minimum height=.55cm] (a) at (0,0) {$B_0$};
\node[draw,rounded corners,minimum width=.9cm,minimum height=.55cm] (b) at (-1.45,-1) {$B_1$};
\node[draw,rounded corners,minimum width=.9cm,minimum height=.55cm] (c) at (0,-1.05) {$B_2$};
\node[draw,rounded corners,minimum width=.9cm,minimum height=.55cm] (d) at (1.45,-1) {$B_3$};
\node[draw,rounded corners,minimum width=.9cm,minimum height=.55cm] (e) at (0,-2.1) {$B_4$};
\draw[thick] (a)--(b); \draw[thick] (a)--(c); \draw[thick] (a)--(d); \draw[thick] (c)--(e);
\end{scope}
\begin{scope}[xshift=5.7cm]
\node at (0,1.75) {an annular necklace};
\foreach \i/\ang in {1/90,2/18,3/-54,4/-126,5/162}{
  \node[draw,rounded corners,minimum width=.75cm,minimum height=.5cm] (n\i) at ({1.25*cos(\ang)},{1.25*sin(\ang)}) {$B_\i$};
}
\draw[thick] (n1)--(n2)--(n3)--(n4)--(n5)--(n1);
\draw[gray] (0,0) circle (.63); \draw[gray] (0,0) circle (1.9);
\end{scope}
\end{tikzpicture}
\caption{A tree-like shadow has a tree as block-adjacency graph.  A tree--necklace shadow allows separated annular cyclic blocks.}
\label{fig:block-graphs}
\end{figure}
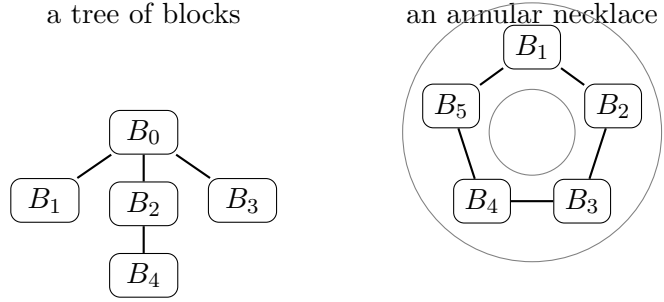

\section{Coorientations and admissibility}

On an immersed arc without inflection points there is a well-defined convex side.  Thus an inflection-free realization of a curve induces a coorientation on every side of every building polygon.

\begin{definition}
A \emph{local coorientation} of a shadow is a choice, for each side of each building polygon, of one of its two local sides.  A side is called \emph{outward} if the chosen side points away from the domain bounded by the corresponding building polygon, and \emph{inward} otherwise.
\end{definition}

\begin{figure}[t]
\centering
\begin{tikzpicture}[scale=.9,>=Stealth]
\coordinate (A) at (-1.5,-.8);
\coordinate (B) at (1.5,-.8);
\coordinate (C) at (0,1.1);
\draw[very thick] (A) .. controls (-.7,-1.05) and (.7,-1.05) .. (B)
                 .. controls (1.7,.25) and (.65,1.15) .. (C)
                 .. controls (-.65,1.15) and (-1.7,.25) .. (A);
\node at (0,-.05) {$B$};
\draw[->,thick] (0,-.9)--(0,-1.45) node[below] {outward};
\draw[->,thick] (1.18,.25)--(1.72,.48);
\draw[->,thick] (-1.18,.25)--(-1.72,.48);
\begin{scope}[xshift=5.0cm]
\coordinate (D) at (-1.2,-.55); \coordinate (E) at (1.2,-.55); \coordinate (F) at (0,.95);
\draw[very thick] (D) .. controls (-.4,-.75) and (.4,-.75) .. (E)
                 .. controls (1.35,.25) and (.55,.9) .. (F)
                 .. controls (-.55,.9) and (-1.35,.25) .. (D);
\node at (0,-.03) {$B'$};
\draw[->,thick] (0,-1.18)--(0,-.65) node[midway,right] {inward};
\draw[->,thick] (.9,.2)--(.35,.05);
\draw[->,thick] (-.9,.2)--(-.35,.05);
\end{scope}
\end{tikzpicture}
\caption{A coorientation records the convex side of each inflection-free side of a building polygon.}
\label{fig:coorientations}
\end{figure}
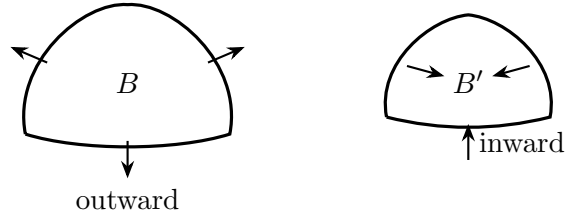

The following admissibility conditions are those appearing in inflection criterion \cite{ShapiroTree}.

\begin{definition}\label{def:admissible}
A local coorientation is called \emph{admissible} if:
\begin{enumerate}[label=(\alph*)]
\item every building $1$-gon is outward cooriented;
\item every building $2$-gon has at least one outward cooriented side;
\item if a building $k$-gon, $k\ge 3$, has all its sides inward cooriented, then the domain bounded by it contains at most $k-3$ neighboring building polygons.
\end{enumerate}
\end{definition}

The first two conditions exclude the obvious local impossibilities.  The third one is the only genuinely polygonal condition: a curvilinear $k$-gon whose sides are convex inward cannot carry too many mutually independent attachments in its interior.

\section{The tree-like realization theorem}

We now record the tree-like realization theorem of \cite{ShapiroTree} in the form needed below.

\begin{theorem}[inflectionless criterion]\label{thm:shapiro}
Let $c$ be a tree-like generic plane curve.  Then $c$ is carried by an orientation-preserving diffeomorphism of the plane to a curve without inflection points if and only if its decomposition admits an admissible local coorientation.
\end{theorem}

\begin{proof}[Comment on the proof]
The necessity follows from the three local convexity restrictions in Definition~\ref{def:admissible}.  The sufficiency is constructive.  Since the block-adjacency graph is a tree, one starts from an exterior building polygon and realizes neighboring polygons inductively, each time using a sufficiently small affine copy with the prescribed convex sides.  No cyclic compatibility condition appears, because a new block is attached to the already constructed part at only one place.  This is precisely the argument of \cite{ShapiroTree}.
\end{proof}

\section{Normalized inflections}

For the minimization problem it is convenient to use a mild normal form.

\begin{definition}
A generic curve is called \emph{normalized} if all its inflections are placed in arbitrarily small neighborhoods of double points and no open edge contains an inflection.
\end{definition}

\begin{lemma}[Normalization]\label{lem:normalization}
Every generic curve is isotopic, through generic curves, to a normalized one without increasing the number of inflections.
\end{lemma}

\begin{proof}
Cut the curve at its double points.  On each open edge, inflections can be moved by a $C^\infty$-small isotopy along the edge toward one of its endpoints.  If two consecutive inflections on the same edge have opposite changes of convexity, the usual local cancellation of a pair of zeros of curvature removes them.  Repeating this on the finitely many edges leaves all remaining changes of convexity in small vertex neighborhoods.  The isotopy may be chosen inside pairwise disjoint tubular neighborhoods of the edges and hence does not create new double points.
\end{proof}

Thus, in a normalized representative, inflections are detected by mismatches of convex coorientations at double points.

\begin{definition}\label{def:conflict}
Let $C$ be a local coorientation.  At a double point, two sides belonging to two adjacent building polygons are glued.  The double point is called a \emph{conflict} of $C$ if the two induced convex sides are incompatible after this gluing, i.e. if smoothing the two incident arcs with nonzero curvature would force opposite choices of convex side.  Let
\[
 \conf(C)
\]
be the number of conflicts of $C$.
\end{definition}

This definition is purely local and depends only on the two cooriented sides meeting at the corresponding vertex.

\begin{lemma}[Evenness]
\label{lem:evenconflicts}
For every complete local coorientation of a connected shadow, the number $\conf(C)$ is even.  Consequently, all exact formulae below automatically respect the elementary fact that a closed generic curve has an even number of sign-changing inflection points.
\end{lemma}

\begin{proof}
Read the coorientation as a sign attached to each open arc of the Eulerian traversal of the shadow: the sign tells on which side of the oriented arc the convex side lies.  A conflict is exactly a place where this sign changes when one passes through a small neighborhood of a double point in the normalized model.  A circular sequence of signs has an even number of sign changes.  Equivalently, the product of all successive transition signs along the closed parametrizing circle is $+1$.
\end{proof}

\section{A finite formula for tree-like curves}

The realization criterion immediately gives a finite model for the minimization problem.

\begin{theorem}\label{thm:formula}
Let $c$ be a tree-like generic plane curve.  Then
\[
 \mu([c])=
 \min_{C\in\A(c)} \conf(C),
\]
where $\A(c)$ is the finite set of admissible local coorientations of the building polygons of $c$.
\end{theorem}

\begin{proof}
First let $c'$ be a normalized representative of $[c]$.  On every side away from the vertex neighborhoods, the curve has a definite convex side.  Hence $c'$ determines a local coorientation $C$.  The necessity part of Theorem~\ref{thm:shapiro}, applied locally to the inflection-free sides of the decomposition, shows that $C$ is admissible.  Each normalized inflection occurs exactly at a vertex at which the two induced convex sides are incompatible.  Therefore
\[
 \#\{\text{inflections of }c'\}\ge \conf(C)\ge
 \min_{D\in\A(c)}\conf(D).
\]
Taking the minimum over normalized representatives and using Lemma~\ref{lem:normalization} gives the lower bound.

Conversely, take an admissible local coorientation $C$.  The constructive proof of Theorem~\ref{thm:shapiro} realizes all building polygons with the prescribed convex sides.  At every non-conflicting gluing one smooths the corner with curvature of the prescribed sign.  At every conflict, one inserts a single normalized inflection in a sufficiently small neighborhood of the corresponding double point.  Since the block-adjacency graph is a tree, these local choices are independent.  The resulting curve is isotopic to $c$ and has exactly $\conf(C)$ inflection points.  Minimizing over $C$ proves the reverse inequality.
\end{proof}

\begin{corollary}[Inflection-free criterion]\label{cor:zero}
A tree-like generic plane curve is isotopic to an inflection-free curve if and only if it admits an admissible local coorientation without conflicts.
\end{corollary}

\begin{corollary}\label{cor:upper}
If a tree-like curve has $n$ double points, then
\[
 \mu([c])\le n.
\]
\end{corollary}

\begin{proof}
For any admissible local coorientation there is at most one conflict at each double point.  Theorem~\ref{thm:formula} gives the claim.  Existence of at least one admissible coorientation is obtained by choosing all sides outward.
\end{proof}

\section{Dynamic programming on the block tree}

The formula in Theorem~\ref{thm:formula} is finite but can be made more explicit.  Let $T(c)$ be the block-adjacency tree.  Root it at an exterior building polygon.  For a building polygon $B$, let $E(B)$ be the set of sides of $B$ along which descendants are attached.

For each possible coorientation of the side by which $B$ is attached to its parent, define
\[
 F_B(\varepsilon)
\]
to be the minimal number of conflicts in the subtree rooted at $B$, under the condition that the parent side induces boundary state $\varepsilon\in\{+,-\}$ at the root side of $B$.  The values are computed by minimizing over all coorientations of the remaining sides of $B$ satisfying Definition~\ref{def:admissible}, adding the local conflict costs and the already computed values of the children.

\begin{proposition}\label{prop:dp}
For tree-like curves, $\mu([c])$ is computable by dynamic programming on the block tree.  If the maximal number of sides of a building polygon is bounded, the computation is linear in the number of double points.
\end{proposition}

\begin{proof}
Because $T(c)$ is a tree, once the coorientation state at the side connecting a block to its parent is fixed, the choices in the descendant subtrees are independent.  Thus the Bellman recursion described above computes the minimum in Theorem~\ref{thm:formula}.  For a block with $k$ sides there are at most $2^k$ local coorientation choices.  If $k$ is uniformly bounded, the cost per block is bounded by a constant, and the total time is linear in the number of blocks, hence in the number of double points.
\end{proof}

\section{The local lower bound for arbitrary embedded shadows}

We now separate what survives for all embedded planar shadows from what is special to the tree-like case.

\begin{definition}
Let $\Gamma$ be a connected embedded planar $4$-regular shadow equipped with an Eulerian traversal.  Denote by $\A_{\loc}(\Gamma)$ the finite set of local coorientations satisfying the polygonal conditions of Definition~\ref{def:admissible} on every building polygon of $\Gamma$.  Define the \emph{local inflection number}
\[
 \mu_{\loc}(\Gamma)=\min_{C\in\A_{\loc}(\Gamma)}\conf(C).
\]
\end{definition}

\begin{theorem}[Universal local lower bound]\label{thm:locallower}
For every generic immersed curve with embedded shadow $\Gamma$,
\[
 \mu([\Gamma])\ge \mu_{\loc}(\Gamma).
\]
Here $[\Gamma]$ denotes the plane isotopy class of generic immersions with the fixed shadow and fixed Eulerian traversal.
\end{theorem}

\begin{proof}
Take a normalized representative realizing $\mu([\Gamma])$.  Its inflection-free open edges determine local convex-side coorientations on all building polygons.  The same local convexity arguments used in the necessity part of inflection theorem imply the three admissibility conditions on every building polygon, see \cite{ShapiroTree}.  Each normalized inflection produces a conflict at the corresponding double point.  Therefore the number of inflections is at least $\conf(C)$ for the induced $C\in\A_{\loc}(\Gamma)$, and the result follows by minimizing over $C$.
\end{proof}

\section{Cycle holonomy}

Let $G_B(\Gamma)$ be the block-adjacency graph of $\Gamma$.  Its vertices are building polygons and its edges are double points at which two building polygons are glued.  When $G_B(\Gamma)$ is not a tree, a coorientation can be transported around a cycle.

\begin{definition}\label{def:holonomy}
Let $Z$ be an oriented simple cycle in $G_B(\Gamma)$ and let $C\in\A_{\loc}(\Gamma)$.  For an edge $e$ of $Z$, put
\[
 \tau_e(C)=
 \begin{cases}
 +1,&\text{if the two coorientations match across the double point }e,\\
 -1,&\text{if }e\text{ is a conflict.}
 \end{cases}
\]
The \emph{coorientation holonomy} of $C$ around $Z$ is
\[
 \Hol_Z(C)=\prod_{e\in Z}\tau_e(C)\in\{\pm1\}.
\]
We call $C$ \emph{cycle-compatible} if $\Hol_Z(C)=+1$ for every cycle $Z$ in $G_B(\Gamma)$.
\end{definition}

Since $\{\pm1\}$ is abelian, it is enough to check $\Hol_Z(C)=+1$ on any cycle basis of $G_B(\Gamma)$.

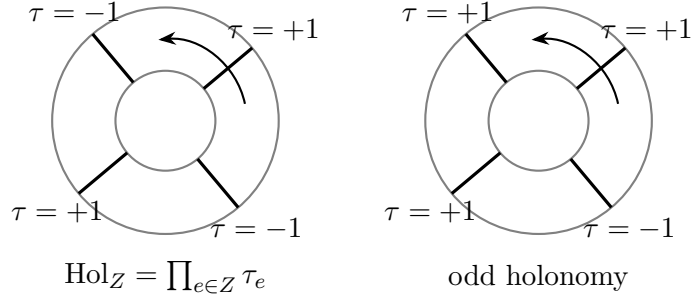
\begin{figure}[t]
\centering
\begin{tikzpicture}[scale=.88,>=Stealth]
\draw[thick,gray] (0,0) circle (.75);
\draw[thick,gray] (0,0) circle (1.7);
\foreach \ang/\lab in {40/{+},130/{-},220/{+},310/{-}}{
  \draw[very thick] (\ang:.75)--(\ang:1.7);
  \node at (\ang:2.12) {$\tau=\lab1$};
}
\draw[->,thick] (12:1.22) arc (12:95:1.22);
\node at (0,-2.35) {$\Hol_Z=\prod_{e\in Z}\tau_e$};
\begin{scope}[xshift=5.6cm]
\draw[thick,gray] (0,0) circle (.75);
\draw[thick,gray] (0,0) circle (1.7);
\foreach \ang/\lab in {40/{+},130/{+},220/{+},310/{-}}{
  \draw[very thick] (\ang:.75)--(\ang:1.7);
  \node at (\ang:2.12) {$\tau=\lab1$};
}
\draw[->,thick] (12:1.22) arc (12:95:1.22);
\node at (0,-2.35) {odd holonomy};
\end{scope}
\end{tikzpicture}
\caption{Convex-side transport around an annular necklace.  The left necklace has trivial holonomy; the right one has odd holonomy and cannot close with only the prescribed local data.}
\label{fig:holonomy}
\end{figure}

\begin{proposition}[Holonomy obstruction]\label{prop:holonomyobstruction}
Let $\Gamma$ be an embedded shadow and let $c$ be a normalized representative with induced local coorientation $C$.  For every cyclic chain of building polygons whose regular neighborhood is an annulus, one has
\[
 \Hol_Z(C)=+1.
\]
Consequently, a prescribed local pattern with odd coorientation holonomy cannot be realized as a closed annular chain with exactly those conflicts.  Correcting the pattern requires changing the parity of the conflict set on that cycle, and therefore costs at least one additional normalized inflection on the cycle.
\end{proposition}

\begin{proof}
Transport the convex side along the annular chain.  At a non-conflicting double point the transported side is unchanged, while at a conflict it changes sign because a single normalized inflection has been inserted.  After one turn around an annulus, the transported side must return to itself.  Hence the product of the transition signs is $+1$.
\end{proof}

\section{An exact extension: tree--necklace shadows}

We now introduce a class which is larger than the tree-like class but still has a controlled global geometry.

\begin{definition}\label{def:necklace}
A connected embedded shadow $\Gamma$ is called a \emph{tree--necklace shadow} if its block-adjacency graph is a cactus and every nontrivial cyclic block is represented geometrically by an annular necklace of building polygons.  More explicitly, each cyclic block is a simple cycle
\[
 B_1,B_2,\ldots,B_m,B_1
\]
such that a small regular neighborhood of $B_1\cup\cdots\cup B_m$ is an annulus, and every component attached to this annulus after deleting the necklace is tree-like.
\end{definition}

Thus a tree-like shadow is the special case in which there are no necklaces.  A tree--necklace shadow may contain many cycles, but the cycles are separated from one another by cut vertices and therefore interact only through tree-like attachments.

Define
\[
 \A_{\cyc}(\Gamma)=\{C\in\A_{\loc}(\Gamma):\Hol_Z(C)=+1
 \text{ for every necklace cycle }Z\}.
\]

\begin{lemma}[Thin annulus realization]
\label{lem:thinannulus}
Let $Z=B_1,\ldots,B_m,B_1$ be one annular necklace in a tree--necklace shadow, and let $C$ be a local coorientation on the necklace satisfying $\Hol_Z(C)=+1$.  Then the necklace has a normalized realization in an arbitrarily thin annulus with exactly the conflicts prescribed by $C$.  The realization may be chosen with arbitrarily small free windows for the tree-like branches attached to the two boundary components of the annulus.
\end{lemma}

\begin{proof}
Choose a round annulus $1-\epsilon<r<1+\epsilon$ and divide it into $m$ narrow angular sectors.  The sides of $B_i$ are represented by short graphs over radial or angular segments inside the $i$th sector.  Since the annulus is thin, the curvature sign of each side is controlled by a small normal perturbation, so the prescribed inward or outward convex side can be achieved independently on all sides except for the closing condition after a full turn.

At a non-conflicting gluing the two neighboring arcs are smoothed with the transported convex side unchanged.  At a conflicting gluing one inserts one standard cubic-shaped transition of curvature sign in a disk of radius $O(\epsilon^2)$ around the double point.  After going around the annulus, the transported convex side returns to its initial value exactly when the product of the transition signs is $+1$, which is the assumed holonomy condition.  Taking $\epsilon$ sufficiently small leaves disjoint boundary windows in which any tree-like attached component can later be inserted.
\end{proof}

\begin{theorem}[Finite formula for tree--necklace shadows]\label{thm:necklaceformula}
Let $\Gamma$ be a tree--necklace shadow.  Then
\[
 \mu([\Gamma])=
 \min_{C\in\A_{\cyc}(\Gamma)}\conf(C).
\]
\end{theorem}

\begin{proof}
The lower bound is Theorem~\ref{thm:locallower} together with the holonomy necessity of Proposition~\ref{prop:holonomyobstruction}.

For the upper bound choose $C\in\A_{\cyc}(\Gamma)$.  Realize every necklace cycle by Lemma~\ref{lem:thinannulus}.  After the necklace cores have been drawn, all remaining attached components are tree-like.  Each is attached to the already realized part at one double point.  The constructive proof of Theorem~\ref{thm:shapiro} then realizes these components inductively, with the same rule: zero cost at compatible gluings and one normalized inflection at each conflict.  Since the block-adjacency graph is a cactus, different necklaces and attached tree branches meet only through cut vertices, and the constructions do not interfere.  The resulting curve is isotopic to the original shadow and has exactly $\conf(C)$ inflection points.  Minimizing over $C$ proves the formula.
\end{proof}

\begin{corollary}\label{cor:necklacezero}
A tree--necklace shadow is isotopic to an inflection-free curve if and only if it has a conflict-free admissible local coorientation whose holonomy is trivial around every necklace cycle.
\end{corollary}

\begin{example}
Consider a single annular necklace with $m$ building polygons and no attached trees.  In the present $\mathbb Z_2$ model the holonomy is simply the parity of the conflict set on the necklace.  Thus a conflict-free assignment automatically has trivial holonomy.  The first genuinely cyclic effect appears in minimization: if the locally optimal admissible assignments have an odd number of necklace conflicts, then no such assignment is realizable with exactly that cost.  One must pass to an admissible assignment with even necklace parity, increasing the cost by at least one.  This parity phenomenon has no analogue in the tree-like case.
\end{example}

\begin{corollary}[One-cycle parity]
\label{cor:onecycle}
Let $\Gamma$ be a tree--necklace shadow with exactly one necklace and no other cyclic block.  Then the minimization in Theorem~\ref{thm:necklaceformula} is the ordinary local minimization with the additional requirement that the number of conflicts on the necklace be even.  In particular, if all locally admissible coorientations force an odd number of necklace conflicts, at least one extra conflict is unavoidable.
\end{corollary}

\begin{proof}
For a single necklace the cycle space is generated by its core cycle.  The condition $\Hol_Z(C)=+1$ is equivalent to the parity equation $\sum_{e\in Z}x_e(C)=0$ over $\Ftwo$, which is exactly evenness of the number of conflicts on the necklace.
\end{proof}

\section{Algorithmic form for bounded cycle rank}

The preceding formula can be computed by the same dynamic programming idea, with a small modification for cycles.

Let $r=b_1(G_B(\Gamma))$ be the cycle rank of the block-adjacency graph.  Choose a feedback edge set $S$ of size $r$; deleting $S$ turns $G_B(\Gamma)$ into a tree.

The holonomy constraints are especially simple over $\Ftwo$.  For a fixed local coorientation put
\[
 x_e(C)=
 \begin{cases}
 0,&\tau_e(C)=+1,\\
 1,&\tau_e(C)=-1.
 \end{cases}
\]
Then the necklace condition is the linear parity equation
\[
 \sum_{e\in Z}x_e(C)=0\pmod 2
\]
for every necklace cycle $Z$.  Thus the extension from tree-like shadows to tree--necklace shadows is not a new nonlinear local problem; it is the tree-like dynamic program together with finitely many mod--$2$ closing equations.

\begin{proposition}\label{prop:feedbackdp}
Fix a class of shadows whose building polygon degrees are bounded by a constant.  For tree--necklace shadows with cycle rank $r$, the number
\[
 \min_{C\in\A_{\cyc}(\Gamma)}\conf(C)
\]
can be computed in time $O(2^r N)$, where $N$ is the number of double points.
\end{proposition}

\begin{proof}
Condition on the coorientation states along the $r$ feedback edges.  After those states are fixed, the remaining block-adjacency graph is a tree, so the Bellman recursion of Proposition~\ref{prop:dp} applies.  The necklace holonomy constraints are checked at the end, or equivalently imposed as parity constraints while processing the feedback edges.  There are $2^r$ choices of feedback states, and each tree computation is linear in $N$ under the bounded-degree assumption.
\end{proof}

\section{Complexity of exact minimization}

Let us spell out the computational problem in a form which avoids analytic ambiguity.  A \emph{combinatorial shadow} consists of the planar rotation system of the embedded $4$-regular graph, the Eulerian traversal coming from the immersed circle, and the incidence list of the building polygons.  The associated decision problem is
\[
 \textsc{MinInflection}=\{(\Gamma,K):\mu([\Gamma])\le K\}.
\]
For arbitrary smooth curves this is not yet a single formal problem until one fixes what a finite certificate for a realization is.  With algebraic arcs and rational inequalities it is more naturally related to the existential theory of the reals than to ordinary graph-theoretic NP.  By contrast, the finite coorientation model considered above is a genuine finite constraint problem.

\begin{proposition}[Certificate for the finite model]\label{prop:inNP}
For the local model $\mu_{\loc}$, and for the tree--necklace model of Theorem~\ref{thm:necklaceformula}, the decision problem
\[
 \min_C \conf(C)\le K
\]
belongs to NP.  More precisely, a certificate consists of one bit of coorientation data for each side of each building polygon.
\end{proposition}

\begin{proof}
Given the coorientation bits, one checks inflection admissibility polygon by polygon.  The number of conflicts is computed by inspecting the two sides meeting at each double point.  For tree--necklace shadows one also checks the parity equations
\[
 \sum_{e\in Z}x_e(C)=0\pmod 2
\]
on a cycle basis of the necklace graph.  All these checks are polynomial in the size of the combinatorial shadow, and the cost $\conf(C)$ is then compared with $K$.
\end{proof}

\begin{corollary}\label{cor:complexitypositive}
For tree-like shadows of bounded face degree the exact number $\mu([\Gamma])$ is computable in polynomial time; indeed Proposition~\ref{prop:dp} gives a linear-time algorithm under this bounded-degree hypothesis.  For tree--necklace shadows of bounded face degree and cycle rank $r$, Proposition~\ref{prop:feedbackdp} gives an $O(2^rN)$ exact algorithm.
\end{corollary}

\begin{remark}
Thus any NP-hardness statement must concern a substantially larger class than bounded-degree tree-like shadows, or must allow the cycle rank or the local polygonal degrees to grow.  This is analogous to bend minimization in graph drawing: bend-minimized orthogonal drawing is NP-hard for planar graphs of maximum degree four, while important restricted cases are polynomial-time or fixed-parameter tractable; see, for instance, \cite{ChangYen}.
\end{remark}

\begin{conjecture}[Complexity of unrestricted shadows]\label{conj:nphard}
For an appropriate purely combinatorial encoding of embedded shadows, the decision problem \textnormal{\textsc{MinInflection}} for unrestricted shadows is NP-hard.  The same should already hold for shadows for which all building polygons have uniformly bounded degree, provided the cycle rank is allowed to grow.
\end{conjecture}

\begin{problem}\label{prob:reduction}
Construct an explicit reduction proving Conjecture~\ref{conj:nphard}.  A natural strategy is to build variable gadgets whose two possible coorientations represent Boolean values and clause gadgets whose unavoidable conflicts measure violated clauses.  The holonomy equations suggest a reduction from a parity-constrained satisfiability problem, while the metric closing obstruction suggests a possible reduction from bend-minimization or orthogonal-drawing problems.
\end{problem}

The value of the present results is therefore twofold.  They provide exact polynomial or fixed-parameter algorithms in the tree-like and tree--necklace regimes, and they isolate the part of the general problem where computational hardness is most likely to enter.

\begin{remark}\label{rem:geometric-complexity}
The invariants considered here are geometric minimization invariants.  They are therefore rather different in nature from the algebraic invariants of generic immersed plane curves introduced by Arnold, such as $J^+$, $J^-$ and $St$; see \cite{ArnoldCaustics,ArnoldPlaneCurves}.  Arnold's invariants are computed from local changes under perestroikas and can be evaluated by explicit combinatorial formulae once the curve is given.  By contrast, the quantities $\mu([\Gamma])$ and $\parnum(\Gamma)$ require an optimization over all geometric realizations of a fixed shadow, together with local convexity, holonomy and possible closing constraints.  This is the reason why their computational complexity is expected to be substantially higher in unrestricted families.
\end{remark}

\section{The Gauss-map multiplicity and parallel tangents}

We now add a second invariant attached to the same class of immersed circles.  Let $c:S^1\to\R^2$ be an oriented immersion and let
\[
 T_c:S^1\to S^1,\qquad T_c(t)=\frac{c'(t)}{|c'(t)|}
\]
be its Gauss, or tangent, map.  Denote its degree by
\[
 \rot(c)=\deg T_c,
\]
the Whitney rotation number.  For a regular tangent direction $u\in S^1$ put
\[
 n_c(u)=\#T_c^{-1}(u),
\]
and define the oriented parallel-tangent number
\[
 \parnum(c)=\min_{u\in S^1\text{ regular}} n_c(u).
\]
For a fixed shadow $\Gamma$ with fixed Eulerian traversal, define
\[
 \parnum(\Gamma)=\min_{c\in[\Gamma]}\parnum(c).
\]
Thus $\parnum(\Gamma)$ is the smallest possible number of appearances of the least-covered tangent direction.

\begin{remark}
If one wants unoriented parallel lines rather than oriented tangent vectors, one should replace $T_c$ by its projectivization $S^1\to\mathbb RP^1$.  Since $\mathbb RP^1\simeq S^1$, this is the map $t\mapsto T_c(t)^2$ in angular coordinates.  The oriented version is slightly cleaner and will be used below.
\end{remark}

\begin{lemma}[Degree lower bound]\label{lem:paralleldegree}
For every oriented immersed circle,
\[
 \parnum(c)\ge |\rot(c)|.
\]
Moreover, for every regular $u\in S^1$ one has
\[
 n_c(u)\equiv \rot(c)\pmod 2.
\]
Consequently,
\[
 \parnum(\Gamma)\ge |\rot(\Gamma)|.
\]
\end{lemma}

\begin{proof}
For a regular value $u$ of the Gauss map, the signed sum of local degrees over $T_c^{-1}(u)$ is $\deg T_c=\rot(c)$.  The unsigned number of preimages is therefore at least $|\rot(c)|$, and it has the same parity as $\rot(c)$.
\end{proof}

The link with inflections is immediate.  If $s$ is arclength and $\theta(s)$ is a local lift of the tangent angle, then
\[
 \frac{d\theta}{ds}=\kappa(s),
\]
where $\kappa$ is the signed curvature.  Hence an ordinary sign-changing inflection is exactly a fold point of the Gauss map: as the curve passes through the inflection, the tangent vector stops moving around $S^1$ in one direction and starts moving in the other direction.

\begin{proposition}[Inflection-free case]\label{prop:parallelconvex}
If $c$ is inflection-free, then
\[
 \parnum(c)=|\rot(c)|.
\]
Consequently, if a shadow $\Gamma$ is inflection-free realizable, then
\[
 \parnum(\Gamma)=|\rot(\Gamma)|.
\]
\end{proposition}

\begin{proof}
If $c$ is inflection-free, the curvature has constant sign.  Thus the Gauss map is a covering map of degree $\rot(c)$, with all sheets oriented in the same direction.  Every regular direction is covered exactly $|\rot(c)|$ times.
\end{proof}

\begin{figure}[t]
\centering
\begin{tikzpicture}[scale=.92,>=Stealth]
\draw[->] (-.2,0)--(6.2,0) node[right] {$s$};
\draw[->] (0,-.25)--(0,3.45) node[above] {$\theta$};
\draw[very thick] plot[smooth] coordinates {(0.15,.35) (1.1,2.4) (1.75,2.85) (2.6,1.15) (3.4,.75) (4.45,2.1) (5.85,3.05)};
\draw[dashed] (0,1.05)--(6.05,1.05) node[right] {$\alpha_1$};
\draw[dashed] (0,2.25)--(6.05,2.25) node[right] {$\alpha_2$};
\fill (1.75,2.85) circle (1.7pt) node[above] {fold};
\fill (3.4,.75) circle (1.7pt) node[below] {fold};
\node[align=center] at (3,-.85) {A lift of the tangent angle.  Inflections are folds of $T_c$,\\and $\parnum(c)$ is the least horizontal intersection number modulo $2\pi$.};
\end{tikzpicture}
\caption{The parallel-tangent invariant is a minimum level multiplicity of the Gauss map.}
\label{fig:gaussfolds}
\end{figure}

For curves with inflections, the invariant $\parnum$ measures something different from the number of inflections.  The number of inflections counts the folds of the Gauss map.  The number $\parnum(c)$ counts how many sheets of the Gauss map necessarily lie over every tangent direction.  A pair of inflections whose Gauss image occupies only a proper arc of $S^1$ need not increase $\parnum(c)$, whereas a pair whose image winds over the whole circle increases the least multiplicity by two.

This gives the following useful reformulation.  Suppose $c$ is generic and let
\[
 s_1,\ldots,s_m
\]
be its inflection points.  On each interval between consecutive $s_i$ the tangent angle is monotone.  The Gauss map is therefore represented by a cyclic collection of oriented intervals on the universal cover of $S^1$.  For a direction $\alpha$,
\[
 n_c(e^{i\alpha})
\]
is the number of these intervals whose projection contains $\alpha$, counted with multiplicity.  Thus $\parnum(c)$ is the minimum covering depth of the circle by the monotonicity intervals of the tangent angle.

\begin{definition}\label{def:gaussload}
Call the integer $\parnum(\Gamma)$ the \emph{Gauss load} of the shadow $\Gamma$.
\end{definition}

\begin{corollary}\label{cor:gaussobstructs}
If
\[
 \parnum(\Gamma)>|\rot(\Gamma)|,
\]
then $\Gamma$ is not inflection-free realizable.  Equivalently,
\[
 \parnum(\Gamma)>|\rot(\Gamma)| \quad\Longrightarrow\quad \mu([\Gamma])>0.
\]
The converse need not hold: inflection folds may be confined to a proper set of tangent directions and hence may fail to increase the least Gauss multiplicity.
\end{corollary}

\begin{remark}
The invariant is also close in spirit to classical formulae involving double tangents.  A direction with many preimages under the projectivized Gauss map produces many pairs of parallel tangencies, and the Fabricius--Bjerre theory relates double tangents, crossings and inflections of a generic closed plane curve; see \cite{FabriciusBjerre,Thompson}.  The present invariant is cruder but directly adapted to minimization over a fixed shadow.
\end{remark}

\begin{problem}[Gauss-load minimization]\label{prob:gaussload}
Develop a finite model for $\parnum(\Gamma)$ analogous to the coorientation model for $\mu([\Gamma])$.  More concretely, augment an admissible coorientation by tangent-angle intervals on the monotone arcs between normalized inflections and minimize the least covering depth of $S^1$ subject to the local polygonal, holonomy and metric closing constraints.
\end{problem}

\begin{conjecture}[Tree--necklace Gauss load]\label{conj:necklacegauss}
For tree-like and tree--necklace shadows, the minimum Gauss load is obtained by a reduced tangent-angle realization of a coorientation minimizing the conflict number, after deleting all redundant full turns of monotonicity intervals.  In particular, in these classes the only way to have $\parnum(\Gamma)>|\rot(\Gamma)|$ should be that every reduced realization has a collection of Gauss folds whose images cover the whole circle.
\end{conjecture}

\section{What remains open beyond the tree--necklace class}

The preceding arguments isolate three distinct obstructions.

\begin{enumerate}[label=(\roman*)]
\item \emph{Local polygonal obstruction.}  This is inflection admissibility condition.  It is present already for tree-like curves.
\item \emph{Cyclic holonomy obstruction.}  This appears when the block-adjacency graph has cycles.  For annular necklace cycles it is exactly the parity condition of Proposition~\ref{prop:holonomyobstruction}.
\item \emph{Gauss-load obstruction.}  Even after the number of folds is known, one still has to understand how their Gauss images cover the tangent circle.  This controls the least possible number of parallel tangencies.
\end{enumerate}

For a fully general embedded shadow there may be a further metric obstruction: even after local admissibility and trivial $\mathbb Z_2$ holonomy, there can be closing conditions involving tangent directions, relative positions and affine parameters of several pieces.  This is why the unrestricted problem is harder than the tree-like one.

\begin{problem}\label{prob:globalconditions}
Find intrinsic global compatibility conditions which, together with local admissibility and cyclic holonomy, characterize inflection-free realizability for arbitrary embedded plane shadows.
\end{problem}

\begin{conjecture}[Cycle-controlled shadows]\label{conj:cyclecontrolled}
For any embedded shadow whose block-adjacency graph is a cactus and whose cyclic blocks are annular in the sense of Definition~\ref{def:necklace}, the formula of Theorem~\ref{thm:necklaceformula} remains valid even when the attached tree-like components have arbitrary nesting, provided that Shapiro's local admissibility conditions are interpreted with respect to the actual nesting of the building polygons.
\end{conjecture}

\begin{problem}\label{prob:gap}
For a general embedded shadow $\Gamma$, estimate the possible gap
\[
 \mu([\Gamma])-\mu_{\loc}(\Gamma).
\]
Is this gap bounded above by a function of $b_1(G_B(\Gamma))$ alone, for example by $2b_1(G_B(\Gamma))$?
\end{problem}

\begin{problem}\label{prob:gauss}
Reformulate the local coorientation, holonomy and Gauss-load obstructions directly in terms of the Gauss diagram of the immersed curve.  Such a formulation would make the connection with Arnold's invariants more transparent and might provide computable lower bounds without first constructing the full planar building-polygon decomposition.
\end{problem}

\section{Concluding remarks}

The tree-like part of the inflection minimization problem has a clean finite form: admissible coorientations are the variables and conflicts are the cost.  This gives both the criterion for inflection-free realizability and an effective algorithm for the minimal normalized number of inflections.

The extension proposed here keeps the useful part of that theory and clarifies what must be added outside the tree-like class.  For arbitrary embedded shadows, the same coorientation model gives a universal lower bound.  For tree--necklace shadows, one obtains an exact formula by adding cyclic holonomy, equivalently a finite system of mod--$2$ equations on necklace cycles.  The algorithmic discussion suggests that exact minimization is tractable for tree-like and bounded-cycle-rank necklace shadows, but probably hard for unrestricted shadows.

The Gauss-map multiplicity gives a complementary invariant.  The inflection number counts the folds of the tangent map, while the Gauss load counts how many sheets of this map must remain over every tangent direction.  Hence the two minimization problems are related but not equivalent.  Understanding their joint finite model may be the most promising next step toward a genuinely global theory of immersed plane-circle shadows.

\end{document}